\documentclass[11pt,twoside]{article}
\usepackage{psfrag}
\usepackage{epsfig}
\usepackage{graphicx}
\usepackage{amsmath}
\usepackage{amssymb}
\usepackage{amsthm}
\usepackage{subfigure}
\usepackage{cite}
\usepackage[latin1]{inputenc}

\newtheorem{lemma}{Lemma}
\newtheorem{theorem}{Theorem}
\newtheorem{corol}{Corollary}

\begin{document}

\vspace*{5mm}

\noindent \textbf{\LARGE Distance Regular Colorings \\ 
of $n$-Dimensional Rectangular Grid }
\flushbottom
\date{}

\vspace*{5mm}
\noindent
\textsc{Sergey V. Avgustinovich} \hfill \texttt{avgust@math.nsc.ru} \\
\textsc{Anastasia Yu. Vasil'eva \footnote{The work of the second author 
is partially supported by the Russian Foundation for Basic Research 
under the grant no. 13-01-00463}} \hfill \texttt{vasilan@math.nsc.ru} \\
{\small Sobolev Institute of Mathematics,
 \\
Novosibirsk State University \\

\medskip

\begin{center}
\parbox{11,8cm}{\footnotesize
\textbf{Abstract.}
We study the infinite graph of $n$-dimensional rectangular grid
that doesn't appear distance regular and the distance regular colorings
of this graph, which are defined as the distance colorings
with respect to completely regular codes.
It is proved that the elements of the parameter
matrix of an arbitrary distance regular coloring form two monotonic sequences.
It is shown that every irreducible distance regular coloring of the
$n$-dimensional rectangular grid has at most $2n+1$ colors.}
\end{center}

\baselineskip=0.9\normalbaselineskip

\section{Introduction} \label{introd}
\label{prelim}
A coloring of the vertices of a graph is perfect if any two vertices of the same color
"see" the same number of vertices of any fixed color.
If in addition the vertices are colored by distance 
from some initial set of vertices then the coloring is distance regular. 
This notion 
is closely related with distance regular graphs. In fact, 
in a distance regular graph the distance coloring 
with respect to an arbitrary vertex is perfect.

Complete classification of the perfect colorings of the 2-dimensional rectangular grid 
into 2, 3 and up to 9 colors can be found in \cite{Ax}, \cite{Puz} and \cite{Krotov-prep} 
respectively. 
All parameters of the distance regular colorings of the 2-dimensional rectangular grid 
were described in \cite{AV}. Parameters of perfect colorings with two colors 
of infinite circulant graphs were studied in \cite{Khor1, Khor2, Par}.
The results of this paper were presented in part in \cite{AV2012}.

For a distance regular
graph the distance partition with respect to an arbitrary vertex is
a perfect coloring; its parameters does not depend on the choice
of the vertex. The completely regular codes in distance regular graphs
are extensively investigated. We study the distance regular colorings in
the graph of $n$-dimensional rectangular grid, which is not distance regular. 
We first prove the monotonicity of the upper and lower diagonals 
of the three-diagonal parameter matrix of a distance regular coloring (Theorem \ref{monot}). 
Then we obtain that the number of colors does not exceed $2n+1$ 
and show that this bound is attainable (Theorem \ref{irr}).

Let's pass to the precise definitions.
A $k$-coloring of the vertices of a graph can be presented as
a function $\varphi$ over the graph vertices with values in the set
$\{1,2,\ldots,k\}$ and as a partition $\{C_1,C_2,\ldots,C_k\}$ of
the graph vertices, where $C_i= \{{\bf x} \ : \ \varphi({\bf x})=i\}, \
i=1,2,\ldots,k$. We do not distinguish between these two interpretations.
A $k$-coloring is \emph{ perfect} (in other
terms, the partition is \emph{ equitable}) with the parameter matrix $A
=(\alpha_{ij})_{k\times k}$ if any vertex of the color $i$ has exactly
$\alpha_{ij}$ adjacent vertices of color $j$ for all $i,j\in\{1,2,...,k\}$.
A~perfect coloring is \emph{ distance regular} if there exists an order
(call, the \emph{standard order}) of the colors such that 
the parameter matrix is three-diagonal with respect to this ordering.
Note that exactly two color ordering of an arbitrary distance regular coloring
are standard: the second one is inverse to the first.
In what follows we suppose that the colors are numbered in the standard order. 
In other words, a perfect coloring is distance regular if the set of
vertices of a color $i, \ i=2,3,\ldots,k,$ \ consists of all
vertices at distance $i-1$ from the set of vertices of the first color.
This notion is intimately related to the notion of completely
regular code. Actually, according to the definition of a
completely regular code \cite{Neumaier}, the vertices of the first color 
(the last color) compose a completely regular code .
When studying codes, we stress at the code distance and cardinality. 
When studying colorings, we regard at the code complement 
and stress at the parameter matrix and structure of all colors.

Denote nonzero elements of the parameter matrix of a distance regular
coloring:

$l_i=\alpha_{i,i-1}  \ (i=2,3,\ldots,k)$ -- the \emph{ lower degree} of the
$i$-th color;

$k_i=\alpha_{i,i}  \ (i=1,2,\ldots,k) $ -- the \emph{ inner degree} of the
$i$-th color;

$u_i=\alpha_{i,i+1}  \ (i=1,2,\ldots,k-1) $ -- the \emph{ upper degree} of
the $i$-th color.

\noindent
In these terms, any vertex of color $i$ "sees" \ $l_i$ vertices of color $i-1$, \
$k_i$ vertices of color $i$ and $u_i$ vertices of color $i+1$; 
obviously $l_i+k_i+u_i=2n$ for any $i$.
We will say that the color $i$ has the \emph{ degree triple}  $(l_i,k_i,u_i)$.

Let ${\bf e^i}, \ i=1,\ldots,n$, \ be the unit vector, 
i.e. $(0,1)$-vector with a unique one at the $i$-th position 
and $S({\bf x})$ be the sphere of radius 1 centered at ${\bf x}$.
Fix an arbitrary distance regular coloring $\varphi$ of $\mathbb{Z}^n$
and an arbitrary vertex ${\bf x}\in \mathbb{Z}^n, \ \varphi({\bf
x})=i$. Let us introduce the following sets of unit vectors:
\begin{eqnarray*}
L_{\varphi}({\bf x}) = \left\{{\bf y}-{\bf x} \ : \
\varphi({\bf y})=i-1, \ {\bf y}\in S({\bf x})\right\} ,\\
I_{\varphi}({\bf x}) = \left\{{\bf y}-{\bf x} \ : \
\varphi({\bf y})=i, \ {\bf y}\in S({\bf x})\right\} , \phantom{vvii} \\
U_{\varphi}({\bf x}) = \left\{{\bf y}-{\bf x} \ : \
\varphi({\bf y})=i+1, \ {\bf y}\in S({\bf x})\right\} .
\end{eqnarray*}
We omit the subscript $\varphi$ if the coloring is clear from the context.
We refer to the vectors in the sets $L_{\varphi}({\bf x}), I_{\varphi}({\bf x}),
U_{\varphi}({\bf x})$ as \emph{ lower, inner} and \emph{ upper directions} 
of the vertex ${\bf x}$ with respect to $\varphi$.
Obviously,
\begin{eqnarray*}
|L_{\varphi}({\bf x})|=l_i, \ \ |I_{\varphi}({\bf x})|=k_i, \ \
|U_{\varphi}({\bf x})|=u_i \ \ \ \ \mbox{and} \\
L_{\varphi}({\bf x})\cup I_{\varphi}({\bf x})\cup
U_{\varphi}({\bf x})= \{\pm {\bf e^i}, \ \ \ i=1,\ldots,n\}.
\end{eqnarray*}

We say that two colorings $\varphi$ and $\psi$ are \emph{ equivalent}
if $\psi$ can be obtained from $\varphi$ by some translate and some color reordering.
In particular, for a distance regular coloring $\varphi$,
the coloring $\psi$ with the inverse order of colors is equivalent
and distance regular; moreover, for an arbitrary vertex ${\bf x}$,
\begin{equation}\label{L=U}
L_{\varphi}({\bf x})=U_{\psi}({\bf x}), \ \ \ \
U_{\varphi}({\bf x})=L_{\psi}({\bf x}), \ \ \ \
I_{\varphi}({\bf x})=I_{\psi}({\bf x}).
\end{equation}

For any set $D$ of directions, use the
notation $-D$ for the set $\{-d \ : \ d\in D\}$.

\section{Reducible colorings} \label{reducible}
Consider colorings of 1-dimensional grid $\mathbb{Z}^1$. 
For an arbitrary $k$, there exist only three nonequivalent
perfect $k$-colorings; they are distance regular and periodical. 
We write theirs periods as sequences of colors:
$$1,2,\ldots,k-1,k,k-1,\ldots,2 ;$$
$$1,1,2,\ldots,k-1,k,k-1,\ldots,2 ;$$
$$1,1,2,\ldots,k-1,k,k,k-1,\ldots,2 .$$

A coloring $\varphi=\varphi(x_1,\ldots,x_n)$ of $\mathbb{Z}^n$ is
called \emph{ reducible} if it can be reduced to the 1-dimensional coloring,
i.e. there exists a $k$-coloring $\varphi_1$ of $\mathbb{Z}^1$ and
$\delta_1,\ldots,\delta_k\in\{0,1,-1\}$ such that
for any $(x_1,x_2,\ldots,x_n)\in \mathbb{Z}^n$,
\begin{equation}\label{nfrom1}
\varphi(x_1,x_2,\ldots,x_n) =
\varphi_1\left(\delta_1x_1+\delta_2x_2+\ldots+\delta_nx_n\right) .
\end{equation}
Readily,
if the coloring $\varphi_1$ of $\mathbb{Z}^1$ is distance regular
then the coloring $\varphi$ defined in accordance with (\ref{nfrom1}) is also distance regular.

The parameter matrix of a perfect
coloring is referred to as \emph{reducible} if it admits a reducible
coloring. All reducible matrices (obtained  by (\ref{nfrom1}))have the form 
$$\left (
\begin{array}{ccccccc}  
2n-\varepsilon_1 r & \varepsilon_1 r& .& .& . &0& 0\\ r & 2n-2r & r & .& . &. & 0\\
 . & . & . & .& . &.& .\\
0 & . & . & .& r &2n-2r & r\\
 0 & 0 & . & .& . &\varepsilon_2 r & 2n-\varepsilon_2 r\end{array} \right) ,$$
 where $r$ equals to the number of nonzero
coefficients $\delta_i$ and $\varepsilon_1,\varepsilon_2\in\{1,2\}$
 \ (the colorings with $(\varepsilon_1,\varepsilon_2)=(1,2)$ and
$(2,1)$ are equivalent).


We call  a coloring $\varphi: \mathbb{Z}^n\longrightarrow\{ 1,\ldots,k\}$ 
\emph{ cylindrical}if it is obtained
from a coloring $\varphi': \mathbb{Z}^m\longrightarrow\{ 1,\ldots,k\}, \ m<n,$
by adding nonessential variables. Clearly, once the initial coloring $\varphi'$ is perfect (distance regular) then the cylindrical coloring is also perfect (respectively, distance regular).

\section{Upper and lower degrees} \label{degrees}
Fix an arbitrary distance regular $k$-coloring $\varphi$ of $\mathbb{Z}^n$.
We are going to prove the monotonicity of the upper degrees 
(and the lower degrees) of the coloring $\varphi$. 
Note that not every distance regular graph possesses this property.

\begin{lemma} \label{incl}
For any $i, \ 1\leq i\leq k-1,$ and any two adjacent vertices ${\bf x}$ and
${\bf y}$ of colors $i$ and $i+1$ respectively we have $L({\bf
x})\subseteq L({\bf y})$ and $U({\bf x})\supseteq U({\bf y})$.
\end{lemma}
\begin{proof} Let $d\in L({\bf x})$; 
then the color of the vertex ${\bf z}={\bf x}+d$ equals $i-1$. 
Since the coloring is distance regular, 
the vertex ${\bf v}\in {\bf y}+d$ has 
the color $i$, and then $d\in L({\bf y})$. \end{proof}

We refer to the sequence of vertices 
${\bf x^1},{\bf x^2},\ldots,{\bf x^r}\in\mathbb{Z}^n \ (r\leq k)$ 
as an \emph{ ascending chain}, if
$\varphi({\bf x^i})=\varphi({\bf x^{i-1}})+1$ and the distance between
${\bf x^{i-1}}$ and ${\bf x^i}$ equals 1 for $i=2,\ldots,r$. 
As a simple consequence of Lemma \ref{incl}, we have
\begin{corol} \label{asc-chain}
Let $r\leq k$, and let ${\bf x^1},{\bf x^2},\ldots,{\bf x^r}$ be an
ascending chain. Then
$$L({\bf x^1})\subseteq L({\bf x^2})\subseteq
\ldots\subseteq L({\bf x^r}) ,$$
$$U({\bf x^1})\supseteq U({\bf x^2})\supseteq
\ldots\supseteq U({\bf x^r}) .$$
\end{corol}
So, we obtain the monotonicity of the lower degrees and the upper degrees:
\begin{theorem} \label{monot}
For an arbitrary distance regular $k$-coloring of $\mathbb{Z}^n$, it holds
$$l_2\geq\ldots\geq l_{k-1}\geq l_k \ \ \ \ \mbox{and} \ \ \ \
u_1\leq u_2\leq\ldots\leq u_{k-1} .$$
\end{theorem}

It follows from Theorem \ref{monot} that the sequence of colors is partitioned
into three segments
$\{1,\ldots, I_1\}, \ \{I_1+1,\ldots,I_2-1\}, \ \{I_2,\ldots,k\}$
(and the parameter matrix, into three layers):

\noindent in the first segment, for every $i\in\{1,\ldots, I_1\}$, it holds $l_i<u_i$,

\noindent in the second one, for every $i\in \{I_1+1,\ldots,I_2-1\}$, it holds $l_i=u_i$,

\noindent in the third, for every $i\in \{I_2,\ldots,k\}$, it holds $l_i>u_i$.

We are going to prove that all lower degrees in the first segment 
are distinct, as well as the upper degrees in the last segment.
\begin{lemma} \label{u<d}
a) If $l_i=l_{i+1}$ then $u_i\leq l_i$ and $-U({\bf x})\subseteq
L({\bf x})$ for any vertex~${\bf x}$ of color $i$. b) If
$u_i=u_{i-1}$ then $l_i\leq u_i$ and $-L({\bf x})\subseteq U({\bf
x})$ for any vertex~${\bf x}$ of color $i$.
\end{lemma}
\begin{proof}
Let $d\in U({\bf x})$. Then the vertex ${\bf x}+d$ has the color
$i+1$. It follows from Lemma \ref{incl} that
$L({\bf x})=L({\bf x}+d)$. Hence, $-l\in L({\bf x}+d)=L({\bf x})$,
and we get a).
Now (\ref{L=U}) gives b).
\end{proof}

The next theorem is a simple consequence of Lemma \ref{u<d}:

\begin{theorem} \label{non-eq-ty}
For any $i<I_1$, it holds $l_i\neq l_{i+1}$;  
for any $i>I_2$, it holds $u_i\neq u_{i-1}$.
\end{theorem}

In conclusion of this section, we state the following.
\begin{corol} \label{coinctriples}
Let $i,j\in\{2,\ldots,k-1\}, \ i\neq j,$ and 
let the degree triples of colors $i$ and $j$ coincide. 
Then the degree triples of all colors between $i$ and $j$ have the form $(a,2b,a)$.
\end{corol}

\section{Colors with the same degree triples} \label{s-triples}

Theorem \ref{non-eq-ty} establishes that 
only the degree triple of form $(a,2b,a)$ can be repeated.
Everywhere until the end of this section we suppose 
$k\geq 4$ and $I_2 > I_1+2$; i.e., there exist repeated degree triples.

\begin{lemma} \label{opp}
Let the colors $i$ and $i+1$ have the same degree triples. Then for
any two adjacent vertices ${\bf x}$ and ${\bf y}$ of colors $i$ and $i+1$,
respectively, we have
$$L({\bf x})= L({\bf y})= -U({\bf x})= -U({\bf y}) ,$$
$$I({\bf x})= I({\bf y})= -I({\bf x})= -I({\bf y}) .$$
\end{lemma}
\begin{proof}
First note that $L({\bf x})= L({\bf y})$ 
and $U({\bf x})= U({\bf y})$ by Lemma \ref{incl}. 
Fix the direction $d\in L({\bf x})$.
Then $d\in L({\bf y})$ by Lemma \ref{incl}, 
this means that $-d\in U({\bf y}+d)$.
Using Lemma~\ref{incl} with our condition $u_i=u_{i+1}$,
we get $-d\in U({\bf y})$. 
Again by Lemma~\ref{incl}, we obtain $-d\in U({\bf x})$.
The second equality follows.
\end{proof}
We emphasize that according to the Lemma \ref{opp}, 
two opposite directions belong or do not belong 
to the set of inner directions simultaneously.

Let us describe the set of vertices with the fixed color $i$. For any set of vertices
$V\subseteq\mathbb{Z}^n$ denote dy $G(V)$ the subgraph of $\mathbb{Z}^n$
generated by $V$.
\begin{lemma} \label{Ci-opp}
Let the colors $i$ and $i+1$ have the same degree triples. Then for
any two vertices ${\bf x}$ and ${\bf y}$ of a connected component
of the graph $G(C_i\bigcup C_{i+1})$, we have
$$L({\bf x})= L({\bf y})= -U({\bf x})= -U({\bf y}) ,$$
$$I({\bf x})= I({\bf y})= -I({\bf x})= -I({\bf y}) .$$
\end{lemma}
\begin{proof}
It is sufficient to prove the equalities for two adjacent vertices 
of the same color $i$ or $i+1$, for example for the color $i$, 
and then apply Lemma \ref{opp}.
Let us show that the second equation holds for every two adjacent
vertices ${\bf x}$ and ${\bf y}$ of color $i$. Suppose that for
some $d\in I({\bf x})$ this direction is not inner for ${\bf y}$;
i.e., without loss of generality, $d\in U({\bf x})$. 
Fig. 1 with the notations of vertices and
Fig. 2 with theirs colors illustrate our reasoning. We reconstruct
the colors in alphabetical order of vertices in Fig. 1 and get the
coloring as in Fig. 2.

\begin{picture}(0,80)(-10,-10)
\multiput(10,0)(25,0){5}{\line(0,1){65}}
\multiput(5,10)(0,15){4}{\line(1,0){110}}
\put(10,25){\makebox(25,15)[c]{$a$}}
\put(35,10){\makebox(25,15)[c]{$c$}}
\put(35,25){\makebox(25,15)[c]{${\bf x}$}}
\put(35,40){\makebox(25,15)[c]{${\bf x}+d$}}
\put(60,10){\makebox(25,15)[c]{$d$}}
\put(60,25){\makebox(25,15)[c]{${\bf y}$}}
\put(60,40){\makebox(25,15)[c]{${\bf y}+d$}}
\put(85,25){\makebox(25,15)[c]{$b$}}
\put(85,40){\makebox(25,15)[c]{$e$}}
\put(35,-15){\makebox(45,15)[c]{Fig. 1}}
\end{picture}
\begin{picture}(0,80)(-150,-10)
\multiput(10,0)(25,0){5}{\line(0,1){65}}
\multiput(5,10)(0,15){4}{\line(1,0){110}}
\put(10,25){\makebox(25,15)[c]{$i$}}
\put(35,10){\makebox(25,15)[c]{$i$}}
\put(35,25){\makebox(25,15)[c]{$i$}}
\put(35,40){\makebox(25,15)[c]{$i$}}
\put(60,10){\makebox(25,15)[c]{$i-1$}}
\put(60,25){\makebox(25,15)[c]{$i$}}
\put(60,40){\makebox(25,15)[c]{$i+1$}}
\put(85,25){\makebox(25,15)[c]{$i$}}
\put(85,40){\makebox(25,15)[c]{$i+2$}}
\put(35,-15){\makebox(45,15)[c]{Fig. 2}}
\end{picture}

\noindent 
Colors $i$ and $i+2$ of the adjacent vertices $ b$ and $e$ 
come in to collision. The first equation follows 
from the second and Lemma \ref{opp}.
\end{proof}

Let us Clarify the structure of connected components of the
graph $G(C_i\bigcup C_{i+1})$.
Let $c\in\mathbb{Z}$ and $\delta=(\delta_1,\delta_2,\ldots,\delta_n)\in\{0,1,-1\}^n$.
Denote by $M(\delta,c)$ the hyperplane
$$M(\delta,c) = \left\{{\bf x} = (x_1,x_2,\ldots,x_n)\in \mathbb{Z}^n \ : \
\delta_1x_1+\delta_2x_2+\ldots+\delta_nx_n = c\right\} .$$

\begin{lemma} \label{subsp}
Let the colors $i$ and $i+1$ have the same degree triples and let $G'$
be a connected component of $G(C_i\bigcup C_{i+1})$. Then there
exist integer $c$ and $(0,1,-1)$-valued vector $\delta=(\delta_1,\ldots,\delta_n)$ such that
$$C_i\cap G'=M(\delta,c), \ \ \ \ \ C_{i+1}\cap G'=M(\delta,c+1).$$
\end{lemma}
\begin{proof} 
Let the repeated degree triple be $(t,2n-2t,t)$.
Let ${\bf v}=(v_1,v_2,\ldots,v_n)\in G'$ be of color $i$. 
Without loss of generality, we suppose
$$U({\bf v}) = \{ {\bf e^1},\ldots,{\bf e^s},-{\bf e^{s+1}},\ldots,-{\bf e^t}\} ,$$
$$L({\bf v}) = \{ -{\bf e^1},\ldots,-{\bf e^s},{\bf e^{s+1}},\ldots,{\bf e^t}\} ,$$
$$I({\bf v}) = \{ \pm {\bf e^{t+1}},\ldots,\pm {\bf e^n}\} .$$
Then define the following constants: \\
$\delta_1 =\ldots=\delta_s=1,$ \\
\ $\delta_{s+1}=\ldots=\delta_t=-1,$ \ \\
$\delta_{t+1}=\ldots=\delta_n=0,$ \ \\
$c=v_1+\ldots+v_s-v_{s+1}-\ldots-v_t.$\\
All neighbors of the vertices of $G'$ belong to the set
$G' \cup M(\delta,c-1)\cup M(\delta,c+2)$.
For an arbitrary vertex ${\bf x}\in G'$, 
one can easily check by induction on distance between ${\bf x}$ and ${\bf v}$ 
that ${\bf x}\in M(\delta,c)$ in case ${\bf x}\in C_i$ 
and  ${\bf x}\in M(\delta,c+1)$ in case ${\bf x}\in C_{i+1}$.
\end{proof}

\begin{theorem} \label{reduc}
Let $\varphi: \mathbb{Z}^n\longrightarrow\{1,2,\ldots,k\}$ be the
distance regular coloring, and let for some $i,j\in \{ 2,\ldots k-2\},$ the
colors $i$ and $j$ have the same degree triples. Then the  degree
triples coincide for all colors from 2 to $k-1$ and the coloring is reducible.
\end{theorem}
\begin{proof}
By Lemma \ref{subsp}, there exists $c\in\mathbb{Z}$ and
$\delta\in \{0,1,-1\}^n$ such that
$M(\delta,c+\varepsilon)\subseteq C_{i+\varepsilon}, \ \varepsilon\in\{0,1\}$.
Then for any $j\in \{1,\ldots,k\}$ by induction on $|j-i|$ one can easily check that
$M(\delta,c+j-i)\subseteq C_{j}$ because the coloring is distance regular. 
In particular, it holds 
$M(\delta,c-i+1)\subseteq C_1, \ \ M(\delta,c+k-i)\subseteq C_k$.
By distance regulatity of the coloring
$M(\delta,c-i)\subseteq C_1$ or $M(\delta,c+j-i)\subseteq C_2$ and $M(\delta,c+k-i)\subseteq C_k$ or $M(\delta,c+k-i)\subseteq C_{k-1}$.
Finally, we find that
$$\varphi(x_1,x_2,\ldots,x_n) =
\varphi'\left(\delta_1x_1+\delta_2x_2+\ldots+\delta_nx_n\right) ,$$
where $\varphi'$ is a distance regular $k$-coloring of $\mathbb{Z}^1$.
\end{proof}

\section{The number of colors} \label{main}
Let us see how to derive distance regular colorings 
of the $n$-dimensional rectangular grid from the colorings 
of the $2n$-dimensional Hamming space ${\bf F}^{2n} = \{ 0,1\}^{2n}$.
Let $g:{\bf F}^{2n}\rightarrow \mathbb{Z}^n$ be the Gray transform; 
i.e., for $\alpha=(\alpha_1,\ldots,\alpha_{2n})\in {\bf F}^{2n}$
$$g(\alpha_1,\ldots,\alpha_{2n}) = (g_0(\alpha_1,\alpha_2),\ldots,g_0(\alpha_{2n-1},\alpha_{2n})), $$
where $g_0(00)=0, \ g_0(01)=1, \ g_0(11)=2, \ g_0(10)=3$.
Let $\psi$ be a coloring of ${\bf F}^{2n}$. 
Define the coloring $\varphi$ of $\mathbb{Z}^n$ as follows:
\begin{equation}\label{psi-g--1}
\varphi(x_1,\ldots,x_n) = \psi\left(g^{-1}(x_1,\ldots,x_n)\right), \ \
(x_1,\ldots x_n)\in \mathbb{Z}^n.
\end{equation}
\begin{lemma} \label{Gray}
Let $\psi$ be a perfect (distance regular) coloring of ${\bf F}^{2n}$. Then
the coloring $\varphi$ of $\mathbb{Z}^n$ defined by (\ref{psi-g--1}) is also perfect
(respectively, distance regular) with the same parameter matrix.
\end{lemma}
\begin{proof}
As the Gray transform preserves the adjacency, the statement 
is straightforward from the definitions.
\end{proof}
We take as $\psi$ the distance coloring of ${\bf F}^{2n}$ 
with respect to the all-zero vertex: 
\begin{equation}\label{5raskr}
\psi({\bf x})=wt({\bf x})+1 , \ \ \  {\bf x}\in{\bf F}^{2n},
\end{equation}
where $wt({\bf x})=
\sum_{i=1}^{2n} x_i$ is the Hamming weight of the vertex ${\bf x}$.
It is distance regular with the parameters 
$l_i=i-1, \ \ u_i=2n-i+1,  \ \ k_i=0, \ \ i=1,2,\ldots,2n+1 .$ 
Then, by Lemma \ref{Gray},  the coloring $\varphi$ is also
distance regular with the same parameters.
The coloring $\varphi$ is not reducible because
its parameter matrix is not reducible.
Moreover, all variables  of $\varphi$ are essential and the coloring
is not cylindrical.

Finally, we can state the main theorem.

\begin{theorem} \label{irr}
For an arbitrary irreducible distance regular $k$-coloring of
$n$-dimensional rectangular grid, 
it holds $k\leq 2n+1$.
An irreducible distance regular $(2n+1)$-coloring exists.
\end{theorem}
\begin{proof} 
By Theorem \ref{reduc}, the coloring is irreducible and 
every two colors have different degree triples.
This means that $I_2-I_1\leq 2$.
Using Theorem \ref{non-eq-ty}, we get $I_1\leq n$ and $k-I_2+1\leq n$.
Finally, $k\leq 2n+1$.
The coloring constructed above (\ref{})
gives us the example of the irreducible $(2n+1)$-coloring.
\end{proof}

\end{document}